\documentclass[11pt]{article}
\usepackage{extpfeil}
\usepackage[all]{xy}
\usepackage{amssymb} 
\usepackage[margin=1in]{geometry}
\usepackage{euscript}
\usepackage{mathtools}
\usepackage{amsthm}
\usepackage{mathrsfs}
\usepackage{url}
\usepackage{color}

\numberwithin{equation}{section}

\theoremstyle{plain}
	\newtheorem{thm}[equation]{Theorem}
	
	\newtheorem{lem}[equation]{Lemma}

	\newtheorem{lem/defn}[equation]{Lemma/Definition}

\theoremstyle{definition}
	\newtheorem{defn}[equation]{Definition}
	\newtheorem{ex}[equation]{Example}
	\newtheorem{question}[equation]{Question}

\theoremstyle{remark}
	\newtheorem{rem}[equation]{Remark}

\def\nc{\newcommand}
\def\on{\operatorname}

\nc{\edit}[1]{\marginpar{\footnotesize{#1}}}

\nc{\C}{\mathbb{C}}
\nc{\Z}{\mathbb{Z}}
\nc{\PP}{\mathbb{P}}
\nc{\R}{\mathbb{R}}
\nc{\LL}{\mathbb{L}}
\nc{\OO}{\mathcal{O}}

\nc{\X}{\EuScript{X}}
\nc{\sZ}{\EuScript{Z}}

\nc{\id}{{\on{id}}}
\nc\Hom{{\on{Hom}}}
\nc\cone{{\on{cone}}}
\nc\Ob{{\on{Ob}}}
\nc\Spec{{\on{Spec}}}
\nc\Mod{{\on{Mod}}}
\nc\Perf{{\on{Perf}}}
\nc\End{{\on{End}}}
\nc{\into}{\hookrightarrow}
\nc{\tr}{\on{tr}}
\nc{\ev}{\on{ev}}
\nc{\im}{\on{im}}
\nc{\Mot}{\on{Mot}}
\nc{\pt}{\on{pt}}
\nc{\coker}{\on{coker}}
\nc{\rk}{\on{rank}}
\nc{\TOP}{\on{Top}_{\mathbb{C}}^{s}}
\nc{\gr}{\on{gr}}
\nc{\Catperf}{\text{Cat}^{\text{perf}}}
\nc{\Sym}{\on{Sym}}
\nc{\xra}{\xrightarrow}
\nc{\lra}{\xleftarrow}
\nc{\Bet}{\mathbf{Betti}_{X}}
\nc{\codim}{\on{codim}}
\nc{\Fred}{\on{Fred}}
\nc{\colim}{\on{colim}}
\nc{\KK}{{\bf K}}
\nc{\Sp}{\on{Sp}}
\nc{\onto}{\twoheadrightarrow}
\nc{\A}{\mathbb{A}}
\nc{\Aff}{\on{Aff}}
\nc{\SH}{\on{SH}}
\nc{\QCoh}{\on{QCoh}}
\nc{\Alg}{\on{Alg}}
\nc{\Br}{\on{Br}}
\nc{\ta}{\widetilde{\a}}
\nc{\Shv}{\on{Shv}}
\nc{\GG}{\mathbb{G}}
\nc{\red}{\color{red}}
\nc{\an}{\on{an}}
\nc{\D}{\on{D}}
\nc{\qc}{\on{qc}}
\nc{\op}{\on{op}}
\nc{\shEnd}{{\mathcal End}}
\nc{\Sph}{\mathbb{S}}
\nc{\Top}{\on{Top}}
\nc{\Map}{\on{Map}}
\nc{\Vect}{\on{Vect}}
\nc{\holim}{\on{holim}}

\newcommand\blfootnote[1]{%
  \begingroup
  \renewcommand\thefootnote{}\footnote{#1}%
  \addtocounter{footnote}{-1}%
  \endgroup
}

\def\Im{\on{Im}}

\def\J{\mathcal{J}}
\def\A{\mathcal{A}}

\def\cG{\mathcal{G}}

\def\top{\on{top}}
\def\a{\alpha}

\def\th{\on{th}}
\def\Perf{\on{Perf}}

\def\Sp{\on{Sp}}

\def\b{\beta}
\def\P{\EuScript{P}}

\def\CC{\EuScript{C}}
\def\Et{\text{Lis-\'Et}}
\def\et{\text{lis-\'et}}

\def\QCoh{\on{QCoh}}

\def\Vect{\on{Vect}}

\title{On the topological $K$-theory of twisted equivariant perfect complexes}
\author{Michael K. Brown \and Tasos Moulinos}
\date{}

\begin{document}
\vspace{18mm} \setcounter{page}{1} \thispagestyle{empty}

\maketitle
\blfootnote{MB gratefully acknowledges support from the National Science Foundation (NSF award DMS-1502553).}

\begin{abstract}
We construct a comparison map from the topological $K$-theory of the dg-category of twisted perfect complexes on certain global quotient stacks to twisted equivariant $K$-theory, generalizing constructions of Halpern-Leistner-Pomerleano and Moulinos (\cite{HLP}, \cite{moulinos}). We prove that this map is an equivalence if a version of the projective bundle theorem holds for twisted equivariant $K$-theory. Along the way, we give a new proof of a theorem of Moulinos that the comparison map is an equivalence in the non-equivariant case. 
\end{abstract}

\tableofcontents

\section{Introduction}

Let $\EuScript{C}at_{\on{dg}}(\C)$ denote the $\infty$-category of $\C$-linear dg categories, and let $\Sp$ denote the $\infty$-category of spectra. Blanc introduced in \cite{blanc} a functor
$$
K^{\top} : \EuScript{C}at_{\on{dg}}(\C) \to \Sp,
$$
the \emph{topological $K$-theory functor for $\C$-linear dg categories}, based on a proposal of To\"en (\cite{toenktop}). Blanc proves that $K^{\top}$ enjoys the following properties:

\begin{itemize}
\item[(1)] $K^{\top}$ maps Morita equivalences to weak equivalences,
\item[(2)] $K^{\top}$ maps short exact sequences of dg categories to fiber sequences, and  
\item[(3)] if $X$ is a separated scheme of finite type over $\C$, there is a natural equivalence
$$
K^{\top}(\Perf(X)) \xra{\simeq} KU(X^{\an})
$$ 
in $\Sp$, where $KU( - )$ is the (ordinary) topological $K$-theory functor for topological spaces, and $X^{an}$ denotes the complex points of $X$ equipped with the analytic topology.
\end{itemize}

A main source of motivation for the construction of $K^{\top}$ is Katzarkov-Kontsevich-Pantev's seminal work \cite{KKP} on noncommutative Hodge theory. The authors predict in \cite[Section 2.2.6(b)]{KKP} that there should be a notion of topological $K$-theory of a $\C$-linear dg-category $\CC$ such that, when $\CC$ is smooth and proper, its topological $K$-theory provides a rational lattice inside its periodic cyclic homology, just as in the classical setting of smooth proper complex varieties. This prediction is formulated precisely in Conjecture 1.7 of \cite{blanc}. 

Blanc's comparison theorem (item (3) above) has been extended by Halpern-Leistner-Pomerleano to quotient stacks. Before stating Halpern-Leistner-Pomerleano's theorem, we recall that, when $X$ is a $\C$-scheme with action of a complex algebraic group $G$, we say $X$ is \emph{$G$-quasi-projective} if there exists a locally
closed immersion $\iota : X\to \PP(V)$, where $V$ is a finite-dimensional $G$-representation and $\iota$ is $G$-equivariant.

\begin{thm}[\cite{HLP} Theorem 3.9]
\label{HLP}
Let $G$ be a complex linear algebraic group, and let $X$ be a smooth $G$-quasi-projective scheme over $\C$. Choose a decomposition 
$$G = U \rtimes H,$$
where $H$ is reductive and $U$ is a connected unipotent group. Let $M$ be a maximal compact subgroup of $H$, and let $KU_M( X^{\an})$ denote the $M$-equivariant topological $K$-theory spectrum of $X^{\an}$. There is a canonical equivalence
$$
\rho_{G, X}: K^{\top}(\Perf([X/G])) \xra{\simeq} KU_M(X^{\an}).
$$
\end{thm}

\begin{rem}
It is important to note that $KU_M(-)$ denotes the \emph{representable} $M$-equivariant complex topological $K$-theory spectrum, as defined in \cite[Chapter XIV]{May}, as opposed to the $K$-theory with compact support discussed in \cite{segal}.
\end{rem}

Blanc's comparison theorem has also been extended by the second author to twisted perfect complexes:

\begin{thm}[\cite{moulinos} Theorem 9.6]
\label{moulinos}
Let $X$ be a separated scheme of finite type over $\C$, and let $\A$ be an Azumaya algebra on $X$. Let $\a$ denote the twist of $K$-theory determined by $\A$, and let $KU^{\a}(X^{\an})$ denote the $\a$-twisted $K$-theory spectrum of $X^{\an}$. There is a canonical equivalence
$$
K^{\top}(\Perf(X, \A)) \xra{\simeq} KU^{\a}(X^{\an}).
$$
\end{thm}

A useful consequence of these comparison theorems is that equivariant topological $K$-theory of smooth $G$-quasi-projective schemes over $\C$ (resp. twisted $K$-theory of separated finite type schemes over $\C$) is invariant under Morita equivalences of dg-categories of equivariant perfect complexes (resp. twisted perfect complexes).

In this article, we study a common generalization of Theorems \ref{HLP} and \ref{moulinos}. That is, we ask whether the topological $K$-theory of the dg-category of twisted equivariant perfect complexes is equivalent to an associated twisted equivariant topological $K$-theory spectrum. While we do not answer this question in this paper, we construct a comparison map from one to the other, and we reduce the question to a projective bundle-type formula in twisted equivariant $K$-theory.

In more detail: let $G$, $X$, and $M$ be as in Theorem \ref{HLP}, $\A$ an Azumaya algebra on the quotient stack $[X/G]$, and $\alpha$ the twist of the $M$-equivariant topological $K$-theory of $X^{\an}$ determined by $\A$. Denote the $\a$-twisted $M$-equivariant $K$-theory spectrum of $X^{\an}$ by $KU_M^{\a}(X^{\an})$. For background on twisted equivariant $K$-theory, we refer the reader to \cite[Section 6]{AS} or \cite{hopkins}. Our first goal is to construct a comparison map
\begin{equation}
\label{comparisonmap}
KU_M^{\a}(X^{\an}) \to K^{\top}(\Perf([X/G], \A)).
\end{equation}
We briefly outline the construction. Let $[P/G] \to [X/G]$ be the Severi-Brauer stack associated to the Azumaya algebra $\A$ (see Section \ref{azumaya} for details). By Theorem \ref{HLP}, we have an equivalence
$$
\rho_{G,P} : K^{\top}(\Perf[P/G])) \xra{\simeq} KU_M(P^{\an}). 
$$
In Theorem \ref{bernardera}, we prove that there is a semi-orthogonal decomposition
$$
\Perf([P/G]) = \langle \Perf([X/G]), \Perf([X/G], \A), \dots, \Perf([X/G], \A^{r-1}) \rangle,
$$
where $r$ is the degree of $\A$, and the Azumaya algebra $\A^j$ is the $j$-fold tensor power of $\A$. It follows that
\begin{equation}
\label{decomp1}
K^{\top}(\Perf([P/G])) \simeq \bigoplus_{j = 0}^{r-1} K^{\top}(\Perf([X/G], \A^j)). 
\end{equation}
We also define in Section \ref{proofsection} a canonical map 
\begin{equation}
\label{decomp2}
\bigoplus_{j = 0}^{r-1} KU^{\a^j}_M(X^{\an}) \to KU_M(P^{\an}).
\end{equation}
When $\a$ is trivializable, this is the equivalence from the projective bundle formula for (untwisted) equivariant $K$-theory (see Theorem \ref{trivializable}(1)). We define the map (\ref{comparisonmap}) by mapping $KU_M^{\a}(X^{\an})$ to $KU_M(P^{\an})$ via (\ref{decomp2}), applying $\rho_{G, P}^{-1}$, and then projecting onto $K^{\top}(\Perf([X/G], \A))$ via (\ref{decomp1}). Our main result is:

\begin{thm}
\label{main}
Let $G$ be a complex linear algebraic group, $X$ a smooth $G$-quasi-projective scheme over $\C$, and $\A$ an Azumaya algebra on $[X/G]$. Let $M$ be as in the statement of Theorem \ref{HLP}, and let $\a$ be the twist of the $M$-equivariant topological $K$-theory of $X^{\an}$ determined by $\A$. There is a natural comparison map 
$$
KU_M^{\a}(X^{\an}) \to K^{\top}(\Perf([X/G], \A)),
$$
and it is an equivalence if the map \eqref{decomp2} is an equivalence.
\end{thm}

We give a proof that the map (\ref{decomp2}) is an equivalence in the case where $G$ is trivial (Theorem \ref{trivializable}(3)). This recovers a result of the second author \cite[Theorem 1.3]{moulinos}; combining this with Theorem \ref{main} gives a new, simpler proof of the comparison theorem for non-equivariant twisted perfect complexes (Theorem \ref{moulinos} above); see Remark \ref{moulinosthm} below for details. We also show in Theorem \ref{trivializable}(2) that the question of whether the map \eqref{decomp2} is an equivalence may be reduced to the case where $X = \Spec(\C)$. 






\begin{rem}
We note that Bergh-Schn\"urer independently obtained the semi-orthogonal decomposition in Theorem \ref{bernardera}; this result appeared in the preprint \cite{BS2} a few days after the first version of this article was posted.
\end{rem}

\vskip\baselineskip
\noindent{\bf Acknowledgements.} The authors are grateful to Benjamin Antieau for valuable comments during the preparation of this article. We thank the referee for his or her careful reading and insightful comments, and especially for catching an error in a previous version of this article.



\section{Background on Azumaya algebras over algebraic stacks}
\label{azumaya}

Let $\X$ be an algebraic stack over a base scheme $S$. We recall the definition of the lisse-\'etale site of $\X$, denoted $\Et(\X)$ (\cite[Definition 9.1.6]{olssonbook}). The objects of $\Et(\X)$ are pairs $(U, f)$, where $U$ is a scheme and $f: U \to \X$ is a smooth morphism. A morphism $(U, f) \to (V, g)$  is given by a morphism $h: U \to V$ of schemes along with a 2-isomorphism $g \circ h \cong f$. A family $\{(U_i, f_i) \to (U, f)\}$ of morphisms is a covering if 
$$
\bigsqcup_i f_i : \bigsqcup_i U_i \to U
$$
is an \'etale covering. Denote by $\X_{\et}$ the lisse-\'etale topos of $\X$. The structure sheaf $\OO_\X \in \X_{\et}$ is given by
$$
(U, f) \mapsto \OO_U. 
$$

\begin{defn}
An \emph{Azumaya algebra of degree $r$ over $\X$} is a locally free $\OO_\X$-algebra $\A$ such that $\A$ is lisse-\'etale-locally isomorphic to $\on{Mat}_r(\OO_\X) := \shEnd_{\OO_\X}(\OO^{\bigoplus r}_\X)$; that is, for every smooth morphism $f: U \to \X$, where $U$ is a scheme, there is an \'etale covering $\{\gamma_i : U_i \to U \}$ such that $(f \circ \gamma_i)^*\A \cong \on{Mat}_r(\OO_{U_i})$. We shall say $\A$ is \emph{trivializable} if there is an isomorphism $\A \cong \shEnd_{\OO_\X}(\mathcal{F})$ for some locally free $\OO_{\X}$-module $\mathcal{F}$. When such an isomorphism has been fixed, we will say $\A$ is \emph{trivial}.
\end{defn} 

We call a morphism $p: \P \to \X$ of algebraic stacks a \emph{Severi-Brauer stack of relative dimension $r$} if it is lisse-\'etale-locally isomorphic to a projectivized vector bundle of relative dimension $r$. We briefly describe a bijection between isomorphism classes of Azumaya algebras of degree $r$ and Severi-Brauer stacks of relative dimension $r - 1$. Define $\on{GL}_r$ to be the group of units in $\on{Mat}_r(\OO_\X)$ (so  $\text{GL}_1 = \mathbb{G}_m $), and set $\on{PGL}_r := \on{GL}_r / \mathbb{G}_m$. Conjugation determines an isomorphism
$$
\phi : \on{PGL}_r \xra{\cong} \A ut(Mat_r(\OO_\X)).
$$
For any group object $\cG$ in $\X_{\et}$, there are cohomology functors $H^{i}(\X, \cG)$ for $i= 0,1$; if $\cG$ is an abelian group object, these functors are defined for all $i \geq 0$. Moreover, the set $H^{1}( \X, \cG)$ classifies $\cG$-torsors on $\X$. The isomorphism $\phi$ therefore gives a bijection between isomorphism classes of Azumaya algebras of degree $r$ on $\X$ and the set $H^{1}(\X, \on{PGL}_{r})$. Such a torsor determines a Severi-Brauer stack of relative dimension $r - 1$. On the other hand, let $p: \P \to \X$ be a Severi-Brauer stack of relative dimension $r-1$. For each smooth morphism $U \to \X$, where $U$ is a scheme, choose an \'etale cover $\{U_i \to U\}$ such that, for each $i$, the pullback of $p$ along $U_i \to U \to \X$ is isomorphic to $\PP^{r-1}_{U_i} \to U_i$. The line bundles $\OO_{\PP^{r-1}_{U_i}}(-1)$ do not necessarily glue to give a line bundle on $\P$; however, applying a construction of Quillen in \cite[Section 8.4]{quillen}, one may construct a canonical vector bundle $\J$ on $\P$ such that, for each $U_i$, the pullback of $\J \to \P \xra{p} \X$ along $U_i \to \X$ is isomorphic to $\OO_{\PP^{r-1}_{U_i}}(-1)^{\oplus r}$. In more detail: one uses that
\begin{itemize}
\item the bundle $\OO_{\PP^{r-1}_\X}(-1)^{\oplus r}$ on $\PP^{r-1}_\X$ is $\on{PGL}_r$-equivariant, and 
\item $\P$ is the bundle associated to a $\on{PGL}_r$-torsor 
\end{itemize}
to construct the descent data necessary to glue the bundles $\OO_{\PP^{r-1}_{U_i}}(-1)^{\oplus r}$ into a bundle $\J$ on $\P$ (Quillen only works with schemes in \cite{quillen}, but the construction adapts to the setting
of algebraic stacks). The Azumaya algebra associated to $\P$ is $p_*(\shEnd_{\OO_\P}(\J))^{\op} = p_*\shEnd_{\OO_\P}(\J^\vee)$.

Given an Azumaya algebra $\A$ on $\X$, denote by $\Perf(\X, \A)$ the dg-category of perfect complexes of left $\A$-modules.

\section{Semi-orthogonal decompositions for Severi-Brauer stacks}
\label{semiorthogonal}
We obtain in this section a semi-orthogonal decomposition of the dg-category of perfect complexes on a Severi-Brauer stack (Theorem \ref{bernardera}), generalizing a theorem of Bernardera (\cite{bern} Theorem 5.1). We emphasize that our proof of Theorem \ref{bernardera} is just a matter of concatenating several results of Bergh-Schn\"urer in \cite{BS}.

Let $\X$ be an algebraic stack over a scheme $S$, and let $p : \P \to \X$ be a Severi-Brauer stack of relative dimension $r -1$, as defined in Section \ref{azumaya}. For each smooth morphism $U \to \X$, where $U$ is a scheme, choose an \'etale cover $\{U_i \to U\}$ such that we have an isomorphism
$$
\xymatrix{
 U_i \times_\X \P \ar[d]^-{p} \ar[r]^-{\varphi_i}_-{\cong} & \PP|_{U_i}^{r-1} \ar[dl] \\
 U_i
}
$$
of schemes over $U_i$ for each $i$.

As discussed in Section \ref{azumaya}, there is a canonical vector bundle $\J$ on $\P$ such that the pullback of $\J \to \P \xra{p} \X$ along $U_i \to \X$ is isomorphic to $\OO_{\PP^{r-1}_{U_i}}(-1)^{\oplus r}$ for all $i$, and
$$
\A : = p_*\shEnd_{\OO_\P}(\J^\vee)
$$ is the Azumaya algebra on $\X$ corresponding to $p$. For all $j \in \Z$, the $j$-fold tensor power $\A^j$ of $\A$ is isomorphic to 
$p_*\shEnd_{\OO_\P}((\J^\vee)^{\otimes j})$; note that there is a canonical isomorphism
$$
p^*(\A^j) \cong \shEnd_{\OO_\P}((\J^\vee)^{\otimes j}).
$$
In particular, $p^*(\A^j)$ is a trivial Azumaya algebra. Noting that $\J^{\otimes j}$ is a right $p^*(\A^j)$-module, we have an equivalence 
$$
T_j: \Perf(\P, p^*(\A^j)) \to \Perf(\P)
$$
given by
$$
\mathcal{F} \mapsto \J^{\otimes j} \otimes_{p^*(\A^j)} \mathcal{F} ,
$$
with inverse $S_j$ given by
$$
\mathcal{G} \mapsto  (\J^\vee)^{\otimes j}  \otimes_{\OO_\P}\mathcal{G} .
$$
Define dg functors
$$
\Phi_j:=T_j \circ p^* : \Perf(\X, \A) \to \Perf(\P).
$$
Note that each $\Phi_j$ has a right adjoint $\Psi_j := p_* \circ S_j$. 

We recall that a dg functor is called \emph{quasi-fully faithful} if the induced functor on homotopy categories is fully faithful.
\begin{thm}
\label{bernardera}
The dg functors $\Phi_j$ are quasi-fully faithful, and there is a semi-orthogonal decomposition
$$
\Perf(\P) = \langle \Im(\Phi_0), \dots, \Im(\Phi_{r-1}) \rangle.
$$
\end{thm}

\begin{proof}

To prove that the $\Phi_j$ are fully faithful, we will apply Bergh-Schn\"urer's ``conservative descent for fully faithfulness" (\cite[Proposition 4.12]{BS}). 
We recall that a functor between ordinary categories is called \emph{conservative} if it reflects isomorphisms. Note that a  triangulated functor is conservative if and only if it reflects zero objects.

Fix $ j \in \Z$. We have diagrams
$$
\xymatrix{
\prod \Perf(\PP^{r-1}_{U_i}) \ar[rr]^-{\prod \varphi_i^*}_-\simeq && \prod \Perf(U_i \times_\X \P ) & \ar[l]  \Perf(\P) 
\\
\prod \Perf(U_i)\ar[u]^-{\prod \Phi_{i,j}} \ar[rr]^-{\prod (- \otimes \OO_{U_i}^{\oplus rj})}_-\simeq && \prod \Perf(U_i, \A^j|_{U_i}) \ar[u]^-{\prod \Phi_{i,j}} & \ar[l] \Perf(\X, \A^j) \ar[u]^-{\Phi_j}
}
$$
and
$$
\xymatrix{
\prod \Perf(\PP^{r-1}_{U_i}) \ar[d]^-{\prod\Psi_{i,j}} \ar[rr]^-{\prod \varphi_i^*}_-\simeq && \prod \Perf(U_i \times_\X \P) \ar[d]^-{\prod\Psi_{i,j}} & \ar[l]  \Perf(\P) \ar[d]^-{\Psi_j}
\\
\prod \Perf(U_i) \ar[rr]^-{\prod(- \otimes \OO_{U_i}^{\oplus rj})}_-\simeq && \prod \Perf(U_i, \A^j|_{U_i})  & \ar[l] \Perf(\X, \A^j),
}
$$
where the products range over each element $U_i \to U$ of each of the \'etale open covers chosen above. The rightmost horizontal maps are given by pullback,
and $\Phi_{i,j}$ (resp. $\Psi_{i,j}$) is the evident analogue of $\Phi_i$ (resp. $\Psi_i$). It's easy to check that the diagrams commute.

The leftmost vertical map $\prod \Phi_{i,j}$ in the first diagram is quasi-fully faithful by the projective bundle theorem for schemes, and therefore the middle vertical map in the first diagram is as well. Since the triangulated functor induced by the dg functor
$$
\Perf(\X, \A^j) \to \prod \Perf(U_i, \A^j|_{U_i})
$$
on the level of homotopy categories reflects 0 objects, it is conservative. It follows from Bergh-Schn\"urer's conservative descent for fully faithfulness that $\Phi_j$ is also quasi-fully faithful. The semi-orthogonal decomposition now follows immediately from Bergh-Schn\"urer's conservative descent theorem for semi-orthogonal decompositions (\cite[Theorem 5.16]{BS}) and their projective bundle theorem for algebraic stacks (\cite[Corollary 6.8]{BS}).

\end{proof}

\section{Halpern-Leistner-Pomerleano's comparison map}
\label{ktop}

In this section, we recall the construction of the equivalence $\rho_{G, X}$ in Halpern-Leistner-Pomerleano's comparison theorem (Theorem \ref{HLP}). Let $X$, $G$, and $M$ be as in the setup of Theorem \ref{HLP}. Let
$$
r : K(\Perf([X/G])) \to KU_M(X^{\an})
$$
denote the comparison map between the connective $G$-equivariant algebraic $K$-theory of $X$ and the (nonconnective) $M$-equivariant topological $K$-theory of $X^{\an}$ (\cite[Section 5.4]{thomason}). We fix some notation: denote by
\begin{itemize}
\item $\Aff_\C$ the category of affine schemes over $\C$,
\item $\Sigma^{\infty}(-)$ the suspension spectrum functor,
\item $( - )_+$ the operation of adjoining a basepoint to a space, and
\item $\underline{\R\Hom}_{\Sp}(-,-)$ the internal mapping object in $\Sp$.
\end{itemize}
The map $r$ induces a morphism 
\begin{equation}
\label{presheaves}
K(\Perf([X/G] \times_\C - )) \to KU_M(X^{\an} \times (-)^{\an} ) \simeq \underline{\R\Hom}_{\Sp}(\Sigma^{\infty}(( - )_+^{\an}), KU_M(X^{\an}))
\end{equation}
of presheaves of spectra on $\on{Aff}_\C$; in the middle term $KU_M(X^{\an} \times (-)^{\an} )$, the input $( - )^{\an}$ is considered as a space with trivial $M$-action. The equivalence on the right follows from \cite[Lemma 3.10]{HLP}. 

Let $\on{Pre}_{\Aff_\C}(\Sp)$ denote the $\infty$-category of presheaves of spectra on $\Aff_\C$, and let $KU-\on{mod}$ (resp. $ku-\on{mod}$) denote the category of $KU$-modules (resp. $ku$-modules). Let
$$
| - | : \on{Pre}_{\Aff_\C}(\Sp) \to \Sp
$$
denote the topological realization functor described in \cite[Definition 3.13]{blanc} (Blanc denotes this functor by $| - |_{\mathbb{S}}$). Given a dg-category $T$ over $\C$, the \emph{semi-topological $K$-theory of $T$}, denoted $K^{\on{st}}(T)$, is defined to be $|K(T \otimes_\C -)|$. As observed in \cite[Definition 3.13]{blanc}, $| - |$ has a right adjoint given by 
$$
E \mapsto \underline{\R\Hom}_{\Sp}(\Sigma^{\infty}(( - )^{\an}_+), E).
$$
The map (\ref{presheaves}) therefore induces a map
\begin{equation}
\label{semi}
K^{\on{st}}(\Perf([X/G])) \to KU_M(X^{\an}).
\end{equation}
As proven in \cite[Section 4]{blanc}, the semi-topological $K$-theory spectrum of any $\C$-linear dg-category $T$ is a $ku$-module. $K^{\top}(T)$ is defined to be $K^{\on{st}}(T) \otimes_{ku} KU$. Noting that (\ref{semi}) is a morphism of $ku$-modules, the adjunction between $KU-\on{mod}$ and $ku-\on{mod}$ given by extension/restriction of scalars yields a map
$$
\rho_{G, X} : K^{\top}(\Perf[X/G]) \to KU_M(X^{\an}). 
$$
This is the map that appears in the Halpern-Leistner-Pomerleano comparison theorem (Theorem \ref{HLP}).

\section{The comparison map}
\label{proofsection}

Let $X$, $G$, $\A$, and $M$ be as in the setup of Theorem \ref{main}. Let $p : [P/G] \to [X/G]$ be the Severi-Brauer stack of relative dimension $r-1$ corresponding to $\A$, and let $\J$ be the vector bundle on $[P/G]$ introduced in Section \ref{azumaya}.

\subsection{Background on invertible algebra bundles and twists of $K$-theory}
\label{inv}
Our reference for this subsection is \cite{freed}. We recall that a \emph{topological groupoid} $Y$ is a pair of topological spaces $Y_0, Y_1$ that form the objects and morphisms, respectively, in a category in which all morphisms are invertible. Denote by $p_0, p_1 : Y_1 \to Y_0$ the source and target maps, respectively. For example, the space $X^{\an}$ equipped with its $M$-action determines the \emph{global quotient groupoid} $[X^{\an}/M]$ with $[X^{\an}/M]_0 = X^{\an}$ and $[X^{\an}/M]_1 = X^{\an} \times M$. In this case, $p_0(x, m) = x$ and $p_1(x, m) = mx$. 

There is a notion of a bundle of invertible (i.e. finite dimensional central simple) $\C$-algebras over a topological groupoid; we refer the reader to \cite[Definition 1.59]{freed} for details. We observe that, if $\mathcal{R}$ is a $G$-equivariant Azumaya algebra over $X$, its analytification $\mathcal{R}_{\top}$ is an invertible algebra bundle over the groupoid $[X^{\an}/M]$. Note that, in \cite[Definition 1.59]{freed}, the fibers in an invertible algebra bundle are allowed to be $\Z/2$-graded, but, in our setting, all invertible algebra bundles will be trivially graded. Given $M$-equivariant Azumaya algebras $\mathcal{R}$ and $\mathcal{R'}$ on $X$, a \emph{morphism} $\mathcal{R}_{\top} \to \mathcal{R}'_{\top}$ of the associated invertible algebra bundles is an $M$-equivariant $\mathcal{R}'_{\top} $-$\mathcal{R}_{\top}$-bimodule. Such a bimodule $\mathcal{B}$ determines an \emph{isomorphism} if there is an $\mathcal{R}_{\top} $-$\mathcal{R}'_{\top}$ bimodule $\mathcal{B'}$ such that $\mathcal{B} \otimes_{\mathcal{R}_{\top}} \mathcal{B'} \cong \mathcal{R}_{\top}'$ and $\mathcal{B}' \otimes_{\mathcal{R}'_{\top}} \mathcal{B} \cong \mathcal{R}_{\top}$.

Twists of topological $K$-theory of a groupoid can be defined in terms of invertible algebra bundles; see \cite[Definition 1.78]{freed} for the precise definition. In particular, each $\A_{\top}^j$ determines a twist $\a^j$ of the $M$-equivariant $K$-theory of $X^{\an}$ (take the local equivalence in \cite[Definition 1.78]{freed} to be the identity on $X^{\an}$). An \emph{isomorphism} of twists of the topological $K$-theory of $[X^{\an}/M]$ arising from invertible algebra bundles on $[X^{\an}/M]$ in this way is just an isomorphism of the invertible algebra bundles, in the sense defined above. We note that an isomorphism of twists determines an equivalence of twisted topological $K$-theory spectra. 

\subsection{On the topological $K$-theory of twisted equivariant projective bundles}
\label{projbundlesection}
Fix $j \in \Z$. Recall that $\A^j  = p_*\shEnd_{[P/G]}((\J^\vee)^{\otimes j})$, and there is a canonical isomorphism 
$$
p^*(\A^j) \cong \shEnd_{[P/G]}((\J^\vee)^{\otimes j}).
$$
The $\OO_{P^{\an}}$-$p^*(\A_{\top}^j)$-bimodule $(\J^{\an})^{\otimes j}$ therefore determines a canonical isomorphism from $p^*(\A_{\top}^j)$ to the trivial invertible algebra bundle $\OO_{P^{\an}}$, and hence an isomorphism from $p^*(\a^j)$ to the zero twist. This isomorphism induces an equivalence
$$
T_j^{\top} : KU_M^{p^*(\a^j)}(P^{\an}) \xra{\simeq} KU_M(P^{\an}).
$$
We define
$$
\phi_{j} := T^{\top}_j \circ p^*: KU^{\a^j}_M(X^{\an}) \to KU_M(P^{\an}).
$$
The map $\phi_j$ is the analogue of the functor $\Phi_j$ on the level of twisted equivariant $K$-theory; note that $\phi_j$ is natural with respect to pullback along morphisms of Severi-Brauer stacks. We have a map
\begin{equation}
\label{projectivebundle}
KU_M(X^{\an}) \oplus KU_M^{\a}(X^{\an}) \oplus \cdots \oplus KU_M^{\a^{r-1}}(X^{\an})
\xra{
    \begin{pmatrix}\phi_0, \dots, \phi_{r-1} \end{pmatrix}
    }
KU_M(P^{\an}).
\end{equation}

\begin{thm}
\label{trivializable}
\text{ }
\begin{enumerate}
\item[(1)] When the twist $\a$ is trivializable, the map \eqref{projectivebundle} recovers the equivalence from the projective bundle formula in (untwisted) equivariant topological $K$-theory. 
\item[(2)] Suppose 
$$
KU_H(*) \oplus KU_H^{\b}(*) \oplus \cdots \oplus KU_H^{\b^{r-1}}(*)
\xra{
    \begin{pmatrix}\phi_0, \dots, \phi_{r-1} \end{pmatrix}
    }
KU_H(P')
$$
is an equivalence whenever $H$ is a closed subgroup of $M$ and $\b$ is an arbitrary $H$-equivariant twist of the $K$-theory of a point; here, $P'$ is the projective $H$-representation associated to $\beta$. In this case, the map \eqref{projectivebundle} is an equivalence.
\item[(3)] The map \eqref{projectivebundle} is an equivalence when $G$ is trivial.
\end{enumerate}
\end{thm}

\begin{rem}
Theorem \ref{trivializable}(3) was proven by the second author in \cite[Theorem 1.3]{moulinos}; it is a consequence of his comparison theorem for the topological $K$-theory of the dg-category of twisted perfect complexes (Theorem \ref{moulinos} above). We give here a direct proof of this fact that does not involve topological $K$-theory of dg-categories. 
\end{rem}

\begin{proof}
Suppose $\a$ is trivializable. We have $\J^{\an} \cong \OO_{P^{\an}}(-1)^{\oplus r}$, and so $\A_{\top}^j$ is isomorphic to the bundle of endomorphisms of $\OO_{X^{\an}}^{\oplus rj}$. Just as above, the $\OO_{X^{\an}}$-$\A_{\top}^j$-bimodule $\OO_{X^{\an}}^{\oplus rj}$ determines an equivalence
$$
U_j^X : KU^{\a^j}_M(X^{\an}) \xra{\simeq} KU_M(X^{\an}).
$$
Similarly, the $\OO_{P^{\an}}$- $p^*(\A_{\top}^j)$-bimodule $\OO_{P^{\an}}^{\oplus rj}$ induces an equivalence
$$
U_j^P : KU^{p^*(\a^j)}_M(P^{\an}) \xra{\simeq} KU_M(P^{\an}).
$$
Now we observe that the diagram
$$
\xymatrix{
KU^{\a^j}_M(X^{\an}) \ar[d]^-{U^X_j} \ar[r]^-{p^*} & KU_M^{p^*(\a^j)}(P^{\an}) \ar[d]^-{U^P_j}  \ar[r]^-{T_j^{\top}} & KU_M(P^{\an}) \\
KU_M(X^{\an}) \ar[r]^-{p^*} & KU_M(P^{\an}) \ar[ru]_-{- \otimes \OO_{P^{\an}}(-j)}
}
$$
commutes on the level of homotopy groups. The commutativity of the square on the left is clear. As for the triangle on the right, since, in the setup of \cite{freed}, compositions of morphisms of twists are given by tensor products of bimodules, it's easy to check that the map $T_j^{\top} \circ (U_j^P)^{-1}$ is induced by the $\OO_{P^{\an}}$- $\OO_{P^{\an}}$-bimodule given by the line bundle $\OO_{P^{\an}}(-j)$. (1) now follows from \cite[Proposition 3.4 (ii)]{hopkins}.


We now prove (2). Since our space $X^{\an}$ is a smooth manifold, $M$ is compact, and $M$ acts smoothly on $X^{\an}$, $X^{\an}$ admits an open cover by $M$-invariant open subsets $\{U_i\}_{i \in I}$ such that each $U_i$ is $M$-equivariantly homotopy equivalent to an $M$-space of the form $M/H_i$ for some closed subgroup $H_i$ of $M$ (\cite[Corollary VI.2.4]{bredon}; see Section IV.1 for the definition of a ``locally smooth" action). Since $X^{\an}$ is second countable, we can assume this cover is countable; write it as $\{U_i\}_{i \ge 1}$. There is an equivalence 
$KU^{\alpha^j}_M(U_i) \xra{\simeq} KU^{\alpha^j}_H(*)$ for each $i$ and $j$ (here, and throughout the proof, we abuse notation slightly: the superscripts ``$\a^j$" really indicate pullbacks of $\a^j$). Covering $P^{\an}$ with the $M$-invariant open sets $V_i := p^{-1}(U_i)$, we obtain a commutative square
\begin{equation}
\label{squareUi}
\xymatrix{
\bigoplus_{j = 0}^{r-1} KU^{\alpha^{j}}_M(U_i)  \ar[rrr]^-
{
 \begin{pmatrix}
    \phi_0 & \cdots & \phi_{r-1}
    \end{pmatrix}
}
 \ar[d]^-{\simeq} &&&  KU_M(V_i) \ar[d]^-{\simeq} \\
\bigoplus_{j =0}^{r-1}  KU^{\alpha^{j}}_{H_i}(*) \ar[rrr]^-
{
 \begin{pmatrix}
    \phi_0 & \cdots & \phi_{r-1}
    \end{pmatrix}
}
 &&& KU_{H_i}(P_i), \\
}
\end{equation}
where $P_i$ is some projective $H_i$-representation. By assumption, the bottom horizontal map is an equivalence; it follows that the top horizontal map is an equivalence as well. 

For $n \ge 1$, let
$$
A_n := U_1 \cup \cdots \cup U_n
$$ 
and 
$$
B_n := V_1 \cup \cdots \cup V_n.
$$
We observe that each $B_n$ is a projective bundle over $A_n$. We have maps
\begin{equation}
\label{anbn}
\bigoplus_{j = 0}^{r-1} KU^{\a^j}(A_n)
\xra{
    \begin{pmatrix}
    \phi_0 & \cdots & \phi_{r-1}
    \end{pmatrix}
} 
KU(B_n)
\end{equation}
for each $n$. Since the top horizontal maps in the squares \eqref{squareUi} are equivalences, the Mayer-Vietoris theorem for twisted $K$-theory (\cite[Section 3]{hopkins}) implies that the maps (\ref{anbn}) are all equivalences. Let
$$
C_{X} := \bigsqcup_n A_n \times [n, n+1]/ \sim
$$
and
$$
C_{P} := \bigsqcup_n B_n \times [n, n+1]/ \sim
$$
denote the ``infinite mapping cylinders", following the terminology of \cite{hopkins}. The proof of \cite[Proposition A.19]{hopkins} implies that, taking the homotopy limit of (\ref{anbn}) over $n$, we get the map
\begin{equation}
\label{C}
\bigoplus_{j = 0}^{r-1} KU^{\a^j}(C_X)
\xra{
    \begin{pmatrix}
    \phi_0 & \cdots & \phi_{r-1}
    \end{pmatrix}
} 
KU(C_P);
\end{equation}
we conclude that this map is also an equivalence.

Let 
$$
g_X : C_X \to X^{\an}, \text{ } g_P : C_P \to  P^{\an}
$$
denote the canonical maps. As in the proof of \cite[Proposition A.19]{hopkins}, choose a partition of unity subordinate to the open cover $\{U_i\}_{i \ge 1}$ of $X^{\an}$, and use it to construct a section $s_X$ of $g_X$. Pulling back along $p : P^{\an} \to X^{\an}$, we get an induced partition of unity subordinate to the open cover $\{V_i\}_{i \ge 1}$ of $P^{\an}$ and therefore an induced section $s_P$ of $g_P$ such that the diagram
\begin{equation}
\label{diagram}
\xymatrix{
P^{\an} \ar[r]^-{s_P} \ar[d]^-{p} & C_P \ar[d] \ar[r]^-{g_P}  & P^{\an} \ar[d]^-{p}\\
X^{\an} \ar[r]^-{s_X} & C_X \ar[r]^-{g_X} & X^{\an}
}
\end{equation}
commutes. The commutativity of diagram (\ref{diagram}) implies that the map (\ref{projectivebundle}) is a section of the equivalence (\ref{C}), and so (\ref{projectivebundle}) is also an equivalence. This proves (2).

Finally, (3) is immediate from (1) and (2), since any twist of the non-equivariant $K$-theory of a point is trivializable.

\end{proof}





\begin{question}
\label{question}
Is the map (\ref{projectivebundle}) an equivalence in general?
\end{question}

\begin{ex}
\label{thepoint}
To answer Question \ref{question}, it suffices, by Theorem \ref{trivializable}(2), to consider the case where $X = \Spec(\C)$. In this case, $P^{\an}$ is simply a projective $M$-representation. Such an object corresponds to a central extension of the form
$$
1 \to S^1 \to \widetilde{M} \to M \to 1.
$$
This central extension canonically determines a complex $\widetilde{M}$-representation $V$, and we have
$$
KU_M(P^{\an}) \simeq KU_M(S(V) / S^1) \simeq KU_{\widetilde{M}}(S(V)),
$$
where $S(V)$ denotes the unit sphere in $V$. Let $B(V)$ denote the unit ball in $V$. By the Thom isomorphism and the long exact sequence of the pair $(B(V), S(V))$, we have an exact sequence
$$
0 \to KU^1_M(P^{\an}) \to \on{Rep}(\widetilde{M}) \xra{\lambda^{-1}[V]} \on{Rep}(\widetilde{M}) \to KU^0_M(P^{\an}) \to 0,
$$
where $\on{Rep}(\widetilde{M})$ denotes the representation ring of $\widetilde{M}$, and $\lambda^{-1}[V] = \sum (-1)^i \lambda^i[V]$. 

On the other hand, by \cite[Example 1.10]{hopkins}, we have
$$
KU^{\alpha^{i}, j}_M(*) = \begin{cases} \on{Rep}^{\alpha^i}(M) & j = 0 \\ 0 & j = 1 \end{cases},
$$
where $\on{Rep}^{\alpha^i}(M)$ denotes the ring of $\alpha^i$-twisted representations of $M$. So, to prove the projective bundle formula for twisted equivariant $K$-theory in the case of a point, one must show
\begin{itemize}
\item[(a)] The map $\bigoplus_{i = 0}^{r-1} \on{Rep}^{\alpha^i}(M) \to \on{Rep}(\widetilde{M}) / (\lambda^{-1}[V])$ induced by (\ref{projectivebundle}) is an isomorphism, and 
\item[(b)] $\ker(\on{Rep}(\widetilde{M}) \xra{\lambda^{-1}[V]} \on{Rep}(\widetilde{M}) ) = 0$.
\end{itemize}
We have been unable to prove either of these two statements. 

We remark that the answer to Question \ref{question} is ``yes" in the case where $X = \Spec(\C)$ and $G$ is a torus. To see this, say $M = (S^1)^d$. We have $H_M^3(* ; \Z) = H^3(BM ; \Z) = H^3((\C\PP^{\infty})^d ; \Z) = 0$; thus, every twist of the $M$-equivariant topological $K$-theory of a point is trivializable, and so the statement follows from Theorem \ref{trivializable}(1). 
\end{ex}

\subsection{Proof of Theorem \ref{main}}

As discussed in the introduction, we define our comparison map
$$
KU^{\alpha}_M(X^{\an}) \to K^{\top}(\Perf([X/G], \A))
$$
by mapping $KU_M^{\a}(X^{\an})$ to $KU_M(P^{\an})$ via (\ref{projectivebundle}), applying the inverse of the Halpern-Leistner-Pomerleano equivalence $\rho_{G, P}$ discussed in Section \ref{ktop}, and then projecting onto the summand $K^{\top}(\Perf([X/G], \A))$ of $K^{\top}(\Perf([P/G]))$ via the equivalence in Theorem \ref{bernardera}.

To prove Theorem \ref{main}, we will need the following 

\begin{lem}
\label{mixed}
Assume that the answer to Question \ref{question} is ``yes". Denote by $\psi_j$ the composition
$$
KU_M(P^{\an}) \xra{(\phi_0 \cdots \phi_{r-1})^{-1}} \bigoplus_{l = 0}^{r-1} KU_M^{\a^l}(X^{\an}) \to KU_M^{\a^j}(X^{\an}),
$$
where the second map is projection onto the $j^{\th}$ component, so that $\psi_j$ is a canonical splitting of $\phi_j$. In this case, in the setting of Theorem \ref{main}, Halpern-Leistner-Pomerleano's comparison map $\rho_{G,P}$ respects the decompositions of $K^{\top}(\Perf[P/G])$ and $KU_M(P^{\an})$ arising from Theorems \ref{bernardera} and the map \eqref{projectivebundle}, respectively. More precisely, if $0 \le j,k \le r-1$ and $j \ne k$, the map $\psi_j \circ \rho_{G, P} \circ \Phi_k$ is trivial, i.e. it induces the zero map in the stable homotopy category.

\end{lem}

\begin{proof}
Recall the comparison map $r : K(\Perf([P/G])) \to KU_M(P^{\an})$ from Section \ref{ktop}. Our first step is to show

$$
\psi_j \circ r \circ \Phi_k: K(\Perf([X/G], \A^k)) \to KU^{\a^j}_M(X^{\an})
$$
is trivial. The map $r$ factors as 
$$
K(\Perf([P/G])) \xleftarrow{\simeq} K(\Vect_G P) \to K(\Vect^{\top}_M P^{\an}) \to KU_M(P^{\an}),
$$
where $K(\Vect_G(P))$ (resp. $K(\Vect_M^{\top}(P))$) denotes the connective algebraic $K$-theory of the exact category of $G$-equivariant vector bundles on $P$ (resp. $M$-equivariant complex vector bundles on $P^{\an}$). Note that, since $P$ is $G$-quasi-projective, $[P/G]$ has the resolution property; this is why the map $K(\Vect_G P)\to  K(\Perf([P/G]))$ is an equivalence. 

 We have a commutative diagram
$$
\xymatrix{
K(\Perf([P/G])) & \ar[l]_-{\simeq}  K(\Vect_G P) \ar[r] & K(\Vect^{\top}_M P^{\an}) \ar[r] & KU_M(P^{\an}) \\ 
K(\Perf([X/G], \A^k)) \ar[u]_-{\Phi_k} & \ar[l]_-{\simeq}  \ar[r]  K(\Vect_G (X, \A^k)) \ar[u]^-{\Phi_k} & K(\Vect^{\top}_M (X^{\an}, \A_{\top}^k)) \ar[u]^-{\Phi_k}  \ar[r]  & KU_M^{\a^k}(X^{\an})\ar[u]^-{\phi_k}.
}
$$
The middle two vertical maps are given by the same formula as $\Phi_k$, and we abuse notation by referring to them with the same symbol. The horizontal maps are all the canonical ones.
Observing that the composition
$$
KU_M^{\a^k}(X^{\an}) \xra{\phi_k} KU_M(P^{\an}) \xra{\psi_j} KU^{\a^j}_M(X^{\an})
$$
is trivial, we conclude that $ \psi_j \circ r \circ \Phi_k$ is trivial.

We now show $\psi_j \circ \rho_{G, P} \circ \Phi_k$ is trivial. Since $\psi_j \circ r \circ \Phi_k$ is trivial, the composition
\begin{align*}
K(\Perf([X/G], \A^k) \otimes_\C - )  \xra{\Phi_k} K(\Perf([P/G]) \otimes_\C - ) & \to  \underline{\R\Hom}_{\Sp}(\Sigma^\infty(( - )^{\an}_+), KU_M(P^{\an})) \\ 
& \xra{\psi_j} \underline{\R\Hom}_{\Sp}(\Sigma^\infty(( - )^{\an}_+), KU^{\a^j}_M(X^{\an}))
\end{align*}
is a trivial map of presheaves of spectra on $\Aff_\C$ (i.e. the map is trivial pointwise), where the middle map is as constructed in Section \ref{ktop}. It follows that the induced map
$$
K^{\on{st}}(\Perf([X/G], \A^k)) \xra{\Phi_k}  K^{\on{st}}(\Perf([P/G])) \to KU_M(P^{\an}) \xra{\psi_j} KU_M^{\a^j}(X^{\an})
$$
is also trivial (here, $K^{\on{st}}( - )$ denotes the semi-topological $K$-theory functor for dg categories, whose definition is recalled in Section \ref{ktop}).  Finally, we conclude that the induced map
$$
K^{\on{\top}}(\Perf([X/G], \A^k)) \xra{\Phi_k}  K^{\on{\top}}(\Perf([P/G])) \xra{\rho_{G, P}} KU_M(P^{\an}) \xra{\psi_j} KU_M^{\a^j}(X^{\an})
$$
is trivial. 
\end{proof}

\begin{proof}[Proof of Theorem \ref{main}]
Our hypothesis is that the answer to Question \ref{question} is ``yes". Define maps $\psi_j$ as in the statement of Lemma \ref{mixed}. We have maps
$$
\psi_j \circ \rho_{G,P} \circ \Phi_j: K^{\top}(\Perf([X/G], \A^{j})) \to  KU^{\a^j}_M(X^{\an}),
$$
where $\rho_{G,P}$ is Halpern-Leistner-Pomerleano's comparison map. By Theorems \ref{bernardera} and our assumption that the map (\ref{projectivebundle}) is an equivalence, the composition 

\begin{align*}
\bigoplus_{j = 0}^{r-1} K^{\top}(\Perf([X/G], \A^j) 
\xra{\begin{pmatrix}
    \Phi_0 & \cdots & \Phi_{r-1}
    \end{pmatrix}}
 K^{\top}(\Perf([P/G]))& 
\xra{\rho_{G,P}} KU_M(P^{\an})
 \\  
& \xra{
\begin{pmatrix}
\psi_0 \\
\vdots \\
\psi_{r-1}
\end{pmatrix}
}
\bigoplus_{j = 0}^{r-1} KU^{\a^j}_M(X^{\an})
\end{align*}
is an equivalence, i.e. the induced map
$$
\bigoplus_{j = 0}^{r-1} K_*^{\top}(\Perf([X/G], \A^j) \xra{
\begin{pmatrix} \psi_0 \circ\rho_{G, P} \circ\Phi_0 & \cdots & \psi_{0} \circ\rho_{G, P}\circ \Phi_{r-1}  \\ \vdots & \vdots & \vdots
\\ \psi_{r-1}\circ \rho_{G, P}\circ \Phi_{0} & \cdots & \psi_{r-1}\circ \rho_{G, P}\circ \Phi_{r-1} \end{pmatrix} 
} \bigoplus_{j = 0}^{r-1} KU^{\a^j, *}_M(X^{\an})
$$
of $\Z$-graded abelian groups is an isomorphism. By Lemma \ref{mixed}, the off-diagonal entries of this matrix are 0, and so the diagonal entries are all isomorphisms; that is, each comparison map $\psi_j \circ \rho_{G,P} \circ \Phi_j: K^{\top}(\Perf([X/G], \A^{j})) \to  KU^{\a^j}_M(X^{\an})$ is an equivalence. Finally, observe that our equivalence is the inverse of $\psi_1 \circ \rho_{G, P} \circ \Phi_1$.  
\end{proof}



\begin{rem}
\label{moulinosthm}
In the construction of the comparison map, the smoothness assumption is only necessary so that we can (1) use Halpern-Leistner-Pomerleano's comparison map, and (2) choose an equivariantly contractible open cover of $X^{\an}$. In particular, the smoothness assumption is not necessary when $G$ is trivial, as we can just use Blanc's comparison map in this case, and, of course, $X^{\an}$ is locally contractible as well. Similarly, ``quasi-projective" can be replaced with ``separated of finite type" when $G$ is trivial. With this in mind, notice that Theorem \ref{trivializable}(3) and Theorem \ref{main} imply that our comparison map is an equivalence in the non-equivariant case. This gives a new, simpler proof of the second author's comparison theorem for the topological $K$-theory of the dg-category of (non-equivariant) twisted perfect complexes (Theorem \ref{moulinos}).
\end{rem}

\bibliographystyle{amsalpha}
\bibliography{Bibliography}
\end{document}